\documentclass[12pt]{amsart}
\usepackage[margin=1in]{geometry}
\usepackage{amssymb}
\usepackage{amsmath}
\usepackage{latexsym}
\usepackage{appendix}
\usepackage{mathrsfs}
\usepackage{amsthm,stmaryrd}
\usepackage{graphicx}
\usepackage{caption}
\usepackage{dsfont}
\usepackage{hyperref}
\usepackage{enumitem}
\usepackage{wasysym}
\usepackage{ulem}
\usepackage{cancel}
\usepackage{color} 

\usepackage{comment}
\usepackage{xcolor}

\usepackage[utf8]{inputenc}
\usepackage[T1]{fontenc} 
\usepackage{adjustbox}
\usepackage{graphicx}

\usepackage{graphicx,tipa}
\newcommand{\arc}[1]{{%
  \setbox9=\hbox{#1}%
  \ooalign{\resizebox{\wd9}{\height}{\texttoptiebar{\phantom{A}}}\cr#1}}}

\numberwithin{equation}{section}
\newtheorem{theorem}[equation]{Theorem}
\newtheorem{proposition}[equation]{Proposition}

\newtheorem{lemma}[equation]{Lemma}

\newtheorem{corollary}[equation]{Corollary}

\newtheorem{conjecture}{Conjecture}
\newtheorem*{open}{Open Question}

\theoremstyle{definition}
\newtheorem{remark}[equation]{Remark}

\newcommand{\inj}{{\rm inj}}

\newcommand{\R}{\mathbb{R}}
\newcommand{\sS}{\mathbb{S}}

\newcommand{\A}{\mathcal{A}}
\newcommand{\N}{\mathbb{N}}

\newcommand{\pa}{\partial}
\newcommand{\intS}{\partial_I\Omega}
\newcommand{\extS}{\partial_E\Omega}
\newcommand{\Om}{\Omega}
\newcommand{\diam}{{\rm diam}}
\newcommand{\cP}{\mathcal{P}}
\newcommand{\DD}{\mathbb{D}}

\newcommand{\be}{\begin{equation*}}
\newcommand{\ee}{\end{equation*}}
\newcommand{\bff}{\begin{proof}}
\newcommand{\ef}{\end{proof}}
\newcommand{\ben}[1]{\begin{equation}\label{#1}}
\newcommand{\een}{\end{equation}}
\newcommand{\bq}{\begin{eqnarray*}}
\newcommand{\bqn}[1]{\begin{eqnarray}\label{#1}}
\newcommand{\eq}{\end{eqnarray*}}
\newcommand{\eqn}{\end{eqnarray}}
\newcommand{\etap}{\eta^\pa}

\def\sideremark#1{\ifvmode\leavevmode\fi\vadjust{\vbox to0pt{\vss
 \hbox to 0pt{\hskip\hsize\hskip1em                        
  \vbox{\hsize2cm\tiny\raggedright\pretolerance10000         
  \noindent #1\hfill}\hss}\vbox to8pt{\vfil}\vss}}}

\begin{document}
\title[]{Escobar constants of planar domains}

\author{Asma Hassannezhad}
\address
{Asma Hassannezhad: University of Bristol,
School of Mathematics,
Fry Building, Woodland Road,
Bristol BS8 1UG, UK}
\email{asma.hassannezhad@bristol.ac.uk}
\author{Anna Siffert$^{*}$}
\thanks{$^{*}$ Corresponding author}
\address
{Anna Siffert: Westf\"alische Wilhelms-Universit\"at M\"unster,
Fachbereich Mathematik und Informatik der WWU,
Einsteinstrasse 62, 48149 M\"unster, Germany}
\email{asiffert@wwu.de}
\date{\today}

\begin{abstract} 
We initiate the study of the higher order Escobar constants $I_k(M)$, $k\geq 3$, on bounded planar domains $M$.
The Escobar constants $I_k$ of the unit disk and a family of polygons are provided. 

\end{abstract}
\maketitle

\section{Introduction}
\label{intro}

In 1997, Escobar \cite{Esc97} introduced the isoperimetric constant $I_2(M)$ of a Riemannian manifold $(M,g)$ with non-empty boundary.
In terms of this constant, he gave a lower bound for the first non-zero eigenvalue $\sigma_1$ of the Dirichlet-to-Neumann operator.
Recently, this theory has been extended by Hassannezhad and Miclo
\cite{HM17} who introduced analogous isoperimetric constants $I_k(M)$ for any $k\in\N_+$ and used them to provide lower bounds on the higher order eigenvalues $\sigma_k$ of the Dirichlet-to-Neumann operator. 
In what follows, we call the isoperimetric constant $I_k(M)$, where $k\in\N_+$, the \textit{$k$-th Escobar constant} (of $M$).  

\bigskip

\noindent{\textbf{Escobar constants.}}
Let {$M$ be a Riemannian manifold with {Lipschitz} boundary and} $\A(M)$ denote the family of all non-empty open subsets $\Omega$ of $M$ with piecewise smooth (or more generally rectifiable)
boundary  $\partial\Omega$, i.e.
$$  \A(M)=\{\Omega\subset M\,\lvert\,\Omega\neq\emptyset, \partial\Omega\,\,\mbox{piecewise smooth}\}.$$
We use $\A$ instead of $\A(M)$ when there is no possibility of confusion.  For $\Omega\in\A$, we write the boundary of $\Omega$ as 
$\pa\Omega=\intS\cup\extS,$ 
where $$\intS:=\pa\Omega\cap M\quad\mbox{and}\quad \extS:=\pa\Omega\cap \pa M$$ are the so-called \textit{interior boundary} and 
 \textit{exterior boundary} of $\Om$ respectively. 
\begin{figure}[ht]
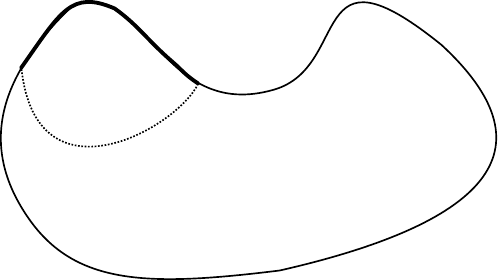
	\caption{Setting.}\label{fig1}
\end{figure}

\smallskip

For $\Omega\in \A$ let 
$\etap(\Om)$ be 
the \textit{isoperimetric ratio} given by
\[
\etap(\Om):=  \frac{|\pa\Omega\cap M|}{|\pa\Omega\cap \pa M|}= \frac{|\intS|}{|\extS|},
\]
where $|\cdot|$ denotes the Riemannian volume of the set. By convention, when $\extS=\emptyset$ then  $\etap(\Om)=\infty$. 
For any $k\in\mathbb{N}$, we introduce the $k$-th Escobar constant as
\bq 
I_k(M)\ :=\ \inf_{(\Om_1,\cdots,\Om_k)\in\A_k}\max_j\etap(\Om_j).\eq
Here, $\A_k=\A_k(M)$ is the family of all  mutually disjoint $k$-tuples $(\Om_1,\cdots,\Om_k)$ such that ${\emptyset\neq}\Om_j\in \A$, $j=1,\cdots,k$. The $k$-th Escobar constant $I_k(M)$ is a scaling invariant quantity.\\

It is easy to check that the definition of $I_2(M)$ does not change if we assume that $(\Omega_1,\Omega_2)\in \A_2$ is a partition of $M$ instead of just a disjoint $2$-tuple. Hence, the definition for $I_2$ given here coincides with the definition given by  Escobar \cite{Esc97}.
{It also has a functional characterisation: 
$$I_2(M)=\inf_{f\in H^1( M)}\frac{\int_M|\nabla f|}{\int_{\partial M} |f-f_{\rm med}|},$$
where $f_{\rm med}$ is the  median of $f$ on $\partial M$:
$$f_{\rm med}:=\sup\left\{t\in\R: \mathcal{H}^{n-1}\{x\in \partial M: f(x)> t\}|> |\partial M|/2\right\}.$$
Here, $\mathcal{H}^{n-1}(\cdot)$ denotes the $(n-1)$-Hausdorff measure of the set.  In other words, $I_2(M)$ gives the geometric description of the optimal constant $C$ in the Poincar\'e trace inequality 
\[\parallel f-f_{\rm med}\parallel_{L^{1}(\partial M)}\le C^{-1} \parallel\nabla f\parallel_{L^1(M)}.\]
See \cite{CFNT17,Esc97}  for more details. However, there is no interpretation for  $I_k$, $k\ge3$,  as the optimal constant of some functional inequalities. }
\smallskip

The $k$-th Escobar constant together with another isoperimetric ratio closely related to the $k$-th Cheeger constant $h_k$ appear in lower bounds for the $k$-th Steklov eigenvalues \cite{HM17}. This has been the primary motivation for the definition of $I_2(M)$ in  \cite{Esc97,Esc99} (see also \cite{Jam15} for a closely related isoperimetric constant)  and $I_k(M)$, $k\ge3$ in \cite{HM17}.
The $k$-th Cheeger constant is given by
\bq h_k(M)\ :=\ \inf_{(\Om_1,\cdots,\Om_k)\in\A_k}\max_j \eta(\Om_j),\eq
where 
\[
\eta(\Om):=    \frac{|\pa\Omega\cap M|}{| \Om|}=\frac{|\intS|}{| \Om|}.\]
 The main motivation for the study of the higher Cheeger constants stems from the fact that they are used for bounding eigenvalues of the Laplace operator\footnote{{Note that there are two slightly different definitions for the well-known Cheeger constant $h_2(M)$ on manifolds with boundary. The choice of the definition depends on the boundary condition. The above choice for $h_2(M)$ is the one related to the second Neumann Laplace eigenvalue, see e.g. \cite{buser}.}}. 
This relationship has been intensively studied  in the literature, see \cite{Ch69,LGT12,Mic15} and the references therein. 

\smallskip

We now restrict to planar domains. 
The study of {the higher-orders Escobar constants is a largely unexplored field.
In this paper, we 
establish fundamental properties of the Escobar constants of planar domains.
More precisely, the aim of this manuscript is  to investigate the relation between $I_k(M)$, $k\ge3$,  and the geometry of a planar domain $M$. We are in particular interested in the behaviour of $I_k(M)$ for $k$ large and the configuration of optimal $k$-tuples.
This may share some similarities with the study of optimal Cheeger clusters, see \cite{BDR12,KL06,BP18} and reference therein for more details. 
A quantity similar to the second Escobar constant $I_2(M)$ also appears in the study of longtime existence result for the curve shortening flow~\cite{Hui98,Gr87}. 
\vspace{0.5cm}

\noindent{\textbf{Main results.}} 
{Escobar \cite{Esc99} and Kusner-- Sullivan \cite{KS98} independently} proved that the unit disk maximizes $I_2$ among all
bounded domains $M$ in $\R^2$ with rectifiable boundary, i.e. \begin{equation}\label{eks}I_2(M)\leq I_2(\DD).\end{equation}
{It also holds in higher dimensions. For the discussion on the higher dimensional version of this inequality we refer to \cite{CFNT17}.}\smallskip

We conjecture that this inequality also holds for higher Escobar constants.
\medskip

\begin{conjecture}\label{conj1}
Let $M\subset\R^2$ be bounded domain with {rectifiable} boundary then for every $k\ge3$
\begin{equation}\label{goal}
I_k(M)\leq I_k(\DD).
\end{equation}
\end{conjecture}
\medskip

We prove this conjecture for $M$ being a polygon in $\R^2$ and $k$ being greater or equal than the number of vertices of $M$. Further, we provide a similar result for a family of curvilinear polygons.
The general case, however, remains open. 
\smallskip

In order to  prove inequality (\ref{goal}) for polygons in $\R^2$
and curvilinear polygons respectively, we first
provide {the value of} $ I_k(\DD)$.

\begin{theorem}[\textbf{Disks}]\label{balls}
The $k$-th Escobar constant of the unit disk $\DD\subset \R^2$ centered at the origin is given by
$$I_k(\DD)=\frac{\sin(\pi/k)}{\pi/k},$$
where $k\in\N$.
\end{theorem}

 Then we can prove inequality (\ref{goal}) for {a regular $n$-polygon $M$  in $\R^2$}  when $k$ is either greater than or equal to the number of vertices of $M$, or $k$ is a divisor of $n$.
\begin{theorem}[\textbf{Regular polygons}]\label{regular-polygons}
Let $n\geq 3$. The regular $n$-gon $D_n$ satisfies
\begin{itemize}
    \item[(i)] 
    the identity
$$I_k(D_n)= \cos\left(\pi/n\right)\le I_k(\DD),$$
for all $k\ge n$; 
\item[(ii)] and 
$$I_k(D_n)=\sin(\pi/k)\cot(\pi/n)k/n\le I_k(\DD),$$
if $n=m k$, where $m\in\N$. 
\end{itemize}
\end{theorem}

The main idea for the proof of $(i)$ is that the domains $\Omega_j$ of the $k$-tuples concentrate at the vertices of $D_n$ -- also see Figure\,\ref{fig}.
Subsequently we provide the inequality (\ref{goal}) for some cases with $k<n$ and $M$ being a regular $n$-gon.

\smallskip

Afterwards, we address the above conjecture for Euclidean $n$-gon. 
. 

\begin{theorem}[\textbf{Euclidean $n$-gon}]\label{intro-polygon} Let $M$ be a Euclidean $n$-gon
and assume that the interior angles of $M$ are ordered as follows

$$0<\theta_1\le\theta_2\le\cdots\le\theta_n<2\pi.$$

Then we have 
$$I_k(M)\le \sin(\theta_1/2)\le \cos\left(\pi/n\right),\qquad\mbox{for all}\,\,\,\, k\ge3.$$
In particular, we have $$I_k(M)\leq I_k(\DD)$$ when either $k\geq n$  or $n\le 4$.
\end{theorem}

The proof is a modification of the proof for Theorem\,\ref{regular-polygons}.
However, as an additional obstacle we no longer have control over the lengths of the edges of the polygons.
We also prove that we have equality $I_k(M)=\sin(\theta_1/2)$ if $k$ is \lq sufficiently large\rq.

\smallskip

Finally, we generalize these considerations to curvilinear polygons. This is done by adapting the proof of Theorem \ref{intro-polygon} in order to find a suitable family of $k$-tuples near the vertix with smallest angle.

\begin{theorem}[\textbf{Curvilinear Polygons}]\label{curvilinear}
Let $M\subset\R^2$ be a curvilinear $n$-gon with at least one interior  angle $<\pi$. Order the interior angles of $M$ in an increasing order 
$$0<\theta_1\le\theta_2\le\cdots\le\theta_n<2\pi.$$ Then we have the inequality $$I_k(M)\le\sin(\theta_1/2),$$ 
{for all $k\ge3$.}
\end{theorem}

\noindent\textbf{Acknowledgments.}
The authors are grateful to anonymous referees for their helpful comments which improved the presentation of the paper. The authors would like to thank the Max Planck Institute for Mathematics in Bonn (MPIM) and the University of Bristol for the hospitality and supporting a research visit of A.~H. and of A.~S. respectively. A.~H. is partially supported by the EPSRC grant EP/T030577/1.
A.~S. gratefully acknowledges the
supports of
the Deutsche Forschungsgemeinschaft (DFG, German Research Foundation) - Project-ID 427320536 - SFB 1442, as well as Germany's Excellence Strategy EXC 2044 390685587, Mathematics M\"unster: Dynamics-Geometry-Structure. 

\section{Elementary properties of the Escobar constants}
\label{sec-ik}
In this section we collect elementary properties of the Escobar constants $I_k(M)$, where $M$ is a bounded domain of $\R^2$ with piecewise smooth boundary $\partial M$.

\smallskip

First we show that the Escobar constants $I_k(M)$ are bounded from above by $1$. For $k=2$ this result has already been stated in \cite{Esc99} in a more general setting. 

\begin{lemma}
For a bounded domain $M$ of $\R^2$ with piecewise smooth boundary $\partial M$,
we have 
\begin{equation*}
  I_k(M)\leq 1  
\end{equation*}
for any $k\in\N_+$.
\end{lemma}
\begin{proof}
Since $I_1(M)=0$, below we consider $I_k(M)$ for $k\geq 2$ only.
Indeed, take a connected smooth arc $U$ of  $\partial M$ with length $|U|=\ell$.
Consider $k$ disjoint connected arcs $B_i\subset U$, $i=1,\cdots,k$, each of length smaller than $\ell/k$, i.e. $|B_i|<\ell/k$.
Define
$$F:U\times [0,\epsilon] \to M, (p,t)\mapsto p-t\nu,$$ where $\nu$ is the unit outward normal direction along $\pa M$ and $\epsilon$ is chosen small enough so that  $F(\cdot,t)$ is a diffeomorphism for any $t\in [0,\epsilon]$. Then we can consider $(\Omega_1^\epsilon,\cdots, \Omega_k^\epsilon)\in\mathcal{A}_k$, where $\Omega_i^\epsilon=F(B_i\times(0,\epsilon))$. Thus, we have $I_k(M)\le{\max_j\eta^\partial(\Omega_i^\epsilon)\to 1}$ when $\epsilon$ tends  to 0.
\end{proof}

\begin{remark}
\label{rem-1}
There is however no universal, positive lower bound on the Escobar constants. Indeed, 
let $$M_\epsilon=[-\epsilon,\epsilon]\times[-1/2,k+1/2]$$ be a thin rectangle and $$\Om_i=(-\epsilon,\epsilon)\times(i-1,i),$$ $i\in\{1,\cdots k\}$ rectangles therein. Then $\etap(\Om_i)=2\epsilon$. Therefore, $I_k(M_\epsilon)$ goes to $0$ as $\epsilon\to0$.
\end{remark}

Next we prove that the Escobar constants $I_k(M)$ are increasing in $k$.
\begin{lemma}
For a bounded domain $M$ of $\R^2$ with piecewise smooth boundary $\partial M$, the Escobar constants are increasing in $k$, i.e. we have $$I_{k+1}(M)\geq I_k(M)$$ for all $k\in\mathbb{N}_+$.
\end{lemma}
\begin{proof}
By definition of $I_{k_0+1}(M)$, there exists a sequence $\{(\Om_1^i,\cdots,\Om_{k_0+1}^i)\}_{i\in\mathbb{N}}$ in $\A_{k_0+1}$ 
satisfying
$$I_{k_0+1}(M)=\lim_{i\rightarrow\infty}
\max_{1\leq j\leq k_0+1}\etap(\Om_j^i).$$
Associate to each element $(\Om_1^i,\cdots, \Om_{k_0+1}^i)\in\A_{k_0+1}$ 
an element $(\tilde{\Om}_1^i,\cdots, \tilde{\Om}_{k_0}^i)\in\A_{k_0}$ by erasing 
exactly one $\Om_l^i$ in $(\Om_1^i,\cdots, \Om_{k_0+1}^i)$ which attains the minimum $\min_{1\leq j\leq k_0+1}\etap(\Om_j^i)$. 
If the minimum is attained by several elements, erase the $\Om_l^i$ with 
the smallest index $l$.
By construction we have
\begin{align*}
I_{k_0}(M)\leq \lim_{i\rightarrow\infty}\max_{1\leq j\leq k_0}\etap(\tilde{\Om_j^i})=\lim_{i\rightarrow\infty}\max_{1\leq j\leq k_0+1}\etap(\Om_j^i)=I_{k_0+1}(M),
\end{align*}
 {which proves the statement.} 
\end{proof}

\section{Bounds on the Escobar constants of planar domains}
\label{sec-boundsr2}

In this section we prove the main results, Theorems\,\ref{balls}--\ref{curvilinear}.

\subsection{Escobar constants $I_k$ for disks}
\label{disks}
We start by providing $I_k(\DD)$, i.e. by proving Theorem\,\ref{balls}.
Note that since $I_k(M)$, $M\subset\R^2$, is invariant under scaling and translation of $M$, this theorem also gives $I_k$ for disks in $\R^2$.
For $k=2$, this result was proven by Escobar \cite{Esc99}.

\begin{proof}[Proof of Theorem\,\ref{balls}]
It is easy to show that 
\begin{equation}\label{inq1}
    I_k(\DD)\le\sin(\pi/k)/(\pi/k).
\end{equation}
Indeed, 
let
$z_j=e^{i(j-1)\frac{2\pi}{k}}$, 
$1\le j\le k$ on $\sS^1=\partial \DD$.
Then the $k$-tuple constituted by the $k$ domains enclosed by the arcs 
\arc{$z_jz_{j+1}$} and the corresponding segments $\overline{z_jz_{j+1}}$, $1\le j\le k$,
consists of $k$ isometric domains $\mathcal{D}$.
Since
$\eta^\pa(\mathcal{D})=\frac{\sin(\pi/k)}{\pi/k}$, this establishes inequality~\eqref{inq1}. 

\smallskip

We now prove that equality holds in (\ref{inq1}).\\  Let $\tilde\A_k(\DD)$ be the collection of all $(\Om_1,\cdots,\Om_k)\in\A_k(\DD)$ satisfying
\begin{enumerate}
    \item\label{1} $\eta^\pa(\Omega_j)\le \frac{\sin(\pi/k)}{\pi/k}$ for each $1\le j\le k$,
    \item each $\Omega_j$ is connected, and 
    \item the interior boundary of each $\Omega_j$ consists of a disjoint union of straight segments.
\end{enumerate}
We first show that
\begin{equation}\label{tildeA} 
I_k(\DD)\ =\ \inf_{(\Om_1,\cdots,\Om_k)\in\tilde\A_k(\DD)}\max_j\etap(\Om_j).\end{equation} 

 \label{6}Assume that $\Omega_j$, for a $j\in\{1,\dots,n\}$, consists of two connected components $C_1$ and $C_2$. If 
$$\frac{|\partial_IC_1|}{|\partial_EC_1|}=\frac{| \partial_IC_2|}{|\partial_EC_2|},$$
we substitute $\Omega_j$ by $C_1$, i.e. by a connected set.
Clearly, this substitution leaves  
$\max_j\etap(\Om_j)$ invariant.
Therefore, let us assume that w.l.o.g. we have 
$$\frac{|\partial_IC_1|}{|\partial_EC_1|}<\frac{| \partial_IC_2|}{|\partial_EC_2|}.$$
Since $a/b<c/d$ implies $a/b<(a+b)/(c+d)<c/d$ for $a,b,c,d\in\mathbb{R}_+$, we can substitute $\Omega_j$ by $C_1$.
This substitution decreases  
$\max_j\etap(\Om_j)$, what is permitted
since in the definition of the Escobar constant the infimum {is taken} over all admissible $k$-tuples.\\
If $\Omega_j$ consists of more than two components, the same argument allows us to substitute  
$\Omega_j$ by its connected component $C_i$ with smallest fraction $\frac{|\partial_IC_i|}{|\partial_EC_i|}$.

\smallskip

Now, assume that the interior boundary of a domain $\Om_j$ is not a straight line.
Then \begin{equation}\label{ah:inq}\etap(\Om_j) \ge \etap(\tilde\Om_j)\end{equation} where $\tilde\Om_j$ has the same exterior boundary as $\Omega_j$ but its interior boundary consists of disjoint straight lines connecting the corresponding endpoints on $\pa_E\Omega_j$. 
By the same argument as in the previous paragraph, we can {simultaneously substitute all $\Om_j$ whose interior boundary  does not consist of union of straight lines with $\tilde\Om_j$. The new collection belongs to $\tilde{\mathcal{A}}_k(\mathbb{D})$ and inequality \eqref{ah:inq} shows that 
$\max_j\etap(\Om_j)$ can decrease with this procedure. 
} The discussion above proves~\eqref{tildeA}.

\smallskip

For $\Omega \in \tilde{\A}(\DD)=\tilde\A_1(\DD)$, we say that $\Omega$ is of \textit{type $n$} when $\extS$ has $n$ connected components. 

\smallskip

Let us first assume that $\Omega_j$ is of type 1. 
Then $$\eta^\pa(\Omega_j)=\frac{2\sin(\alpha_j/2)}{\alpha_j},$$ where $\alpha_j= |\extS_j|$. 
Assume $\alpha_j\le \pi$ and
note that the function $f:\R_+\rightarrow\R$ given by $$f(\alpha)=\frac{2\sin(\alpha/2)}{\alpha}$$ is decreasing on $(0,\pi)$. Since assumption (\ref{1}) can be written as 
$$f(\alpha_j)=\eta^\pa(\Omega_j)\le f(2\pi/k),$$ it implies $\alpha_j\ge\frac{2\pi}{k}$ for $\alpha_j\in(0,\pi)$.\\
If  $\alpha_j=|\extS_j|\ge \pi$, then $\eta^\pa(\Omega_j)=\frac{2\pi-\alpha_j}{\alpha_j}f(2\pi-\alpha_j)\le f(2\pi/k)$.

\smallskip

If $\Omega_j$ is of type $n$, $n\ge2$, then each connected component of $\sS^1\setminus\extS_j$ should contain the exterior boundary of at least one $\Omega_t$ for some $t\ne j$. Otherwise, we can add this arc to the exterior boundary of $\Omega_j$ and remove the corresponding segment in the  interior boundary, see Fig. \ref{typen}. This decreases $\eta^\pa(\Omega_j)$.

\smallskip

\begin{figure}
    \centering
   \includegraphics{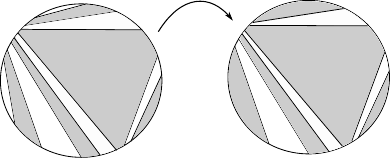}
    \caption{This example shows how one can replace type~2  domains by type~1 domains and decrease $\eta^\pa(\Omega_j),\,j=1,2$.}
    \label{typen}
\end{figure}We now consider the following cases:\medskip

\noindent\textit{Case 1.} All $\Omega_j$, $1\le j\le k$, are of type 1. Then $\alpha_j= |\extS_j|\ge\frac{2\pi}{k}$. 
Since $\sum_{j=1}^k\alpha_j\le 2\pi$, the only possible situation is that the domains $\Omega_j$ are all isometric to $\mathcal{D}$.\medskip

\noindent\textit{Case 2.} There exists at least one $j\in\{1,\dots,k\}$ such that $\Omega_j$ is of type $n\geq 2$. We claim that each connected component of $\sS^1\setminus\extS_j$ contains the boundary of  a type 1 domain $\Omega_k$. Indeed, if a connected component of $\sS^1\setminus\extS_j$ contains the boundary of only one element in the $k$-tuple, then it is clear that this element is of type 1.  By induction, it is easy to show that if it contains the exterior boundary of more than one element, then at least one of them is of type 1 or can be substituted by an element of type 1. This substitution is possible for the same reason discussed above and depicted in Fig. \ref{typen}.
Hence, we get that 
\[|\intS_j|\ge 2\sin(\alpha_1/2)+\dots 2\sin(\alpha_n/2)\ge 2n\sin(\pi/k). \]
Here $\alpha_i$, $i\in\{1,\dots, n\}$, denote the centric angles of the connected components of $\intS_j$.
Since $\eta^\pa(\Omega_j)\le f(2\pi/k)$ by assumption $(1)$, we have $|\extS_j|\ge n\frac{2\pi}{k}$. On the other hand we have $\sum_{j=1}^k|\extS_j|\le 2\pi$. Therefore, it is impossible to have any type except type 1 domains if we want to decrease the value of $\eta^\pa(\cdot)$. Hence, only case $1$ gives us an optimal $k$-tuple. This completes the proof.
\end{proof}

\subsection{Escobar constants $I_k$ for regular $n$-gon and $k\geq n$}
\label{regular-polygons-I}
{In this subsection, we will prove Theorem\,\ref{regular-polygons}.
For the considerations below, it is convenient to introduce the following notations:
Let $\mathcal{K}_n$ denote the set of all convex polygons with $n$ vertices.
For $K\in\mathcal{K}_n$, let $\tilde\A_k(K)$ denote the collection of all $k$-tuple $(\Om_1,\cdots,\Om_k)\in\A_k(K)$ such that \begin{itemize}
    \item[(1)] each $\Om_j$ is connected, and 
    \item[(2)] the closure of  each $\pa_I\Om_j$ made of disjoint union of straight segments.
\end{itemize}
We say that a domain $\Omega\in\tilde{\mathcal{A}}(K):=\tilde{\mathcal{A}}_1(K)$ is of type $\ell$, if
$\pa_E\Omega$ has components in exactly $\ell$ edges of $K$.
\begin{lemma}
\label{lemma-type}
Let $K\in\mathcal{K}_n$. Then there are at most $n-2$  domains $\Omega_j\in\tilde{\mathcal{A}}(K)$ which are of type $\ell$ with $\ell\geq 3$ and  $(\Om_1,\dots,\Om_{n-2})\in\tilde\A_{n-2}(K).$
\end{lemma}}
\begin{proof}
{We prove the claim by induction. Let $n=3$, i.e. $K\in\mathcal{K}_3$.
Assume that there exists a domain $\Omega_1\in\tilde{\mathcal{A}}(K)$ of type $3$.
Then $K\setminus\Omega_1$  consists of connected sets which contain components of at most two edges of $K$. Hence, there can be no other domain of type greater or equal than three. This settles the induction beginning.}

\smallskip 

{We now assume that the claim holds true for all $n\leq n_0$, where $n_0$ is some integer with $n_0\geq 3$.
Consider a set $K\in\mathcal{K}_{n_0+1}$ and suppose that there are at least $n_0$ domains $\Omega_j\in \tilde{\mathcal{A}}(K)$ of type at least three. Then there exists a $\Omega_{j_0}$, such that one connected component of $K\setminus\Om_{j_0}$ contains all the remaining $\Omega_j$ of type greater or equal than three.
This is indeed the case:}
{Consider a fixed $\Omega_{k_1}$ of type greater or equal than three and denote
the connected components of $K\setminus\Omega_{k_1}$ by $C^1_{1},\dots,C^1_{m_1}$. If one of these connected components contains all the remaining $\Omega_j$ of type equal to or greater than three, the claim follows. Otherwise there exists at least two $C^1_i$ which contain  sets of such types.
W.l.o.g. we assume this to be the sets $C^1_1$ and $C^1_2$.
Let $\Omega_{k_2}\in C^1_1$ and denote the connected components of $C^1_1\setminus\Omega_{k_2}$ by $C^2_{1},\dots,C^2_{m_2}$.  
If one of these connected components contains all the remaining $\Omega_j$ of type equal to or greater than three, the claim follows. Otherwise we repeat this process finitely many times to find $\Om_{j_0}$.\\
Due to the existence of $\Omega_{j_0}$, there exists an edge which is disjoint from the remaining $\Omega_j$. We now collapse this edge to one point.
Note that this  neither changes the number of the remaining $\Omega_j$ of type greater or equal than three nor does this procedure change their type. Furthermore,  $\Omega_j$ stay disjoint.
Hence, this process leaves us with a set in $\mathcal{K}_{n_0}$ which has at least $n_0-1$ domains $\Omega_j$ of type at least three. This, however, yields a contradiction to the induction assumption and hence the claim is established.
}
\end{proof}

We now prove Theorem\,\ref{regular-polygons}, i.e. we provide $I_k$ for regular $n$-gons with $n\ge 3$, which will henceforth be denoted by $D_n$.
Since the Escobar constants $I_k(D_n)$ are invariant under translation and scaling of $D_n$, we assume w.l.o.g. that
$D_n$ is centered at the origin and its vertices lie on the unit circle. 
\begin{proof}[\textbf{Proof of Theorem\,\ref{regular-polygons}}]
We start by proving ${(i)}$.
We first prove the inequality 
\begin{equation}
\label{ineq-1}
I_k(D_n)\leq\cos\left(\pi/n\right).
\end{equation}
It is sufficient to 
construct a $k$-tuple $(\Omega_1,\cdots,\Om_k)\in \A_k(D_n)$ such that
$\max\eta^{\pa}(\Omega_j)\leq\cos\left(\pi/n\right).$
Recall that for a regular convex $n$-gon, each interior angle is given by
$\theta=\pi-2\pi/n$.
Near a fixed corner of $D_n$, consider the $k$-tuple $(\Omega_1,\cdots,\Om_k)$ depicted in Figure\,\ref{fig}, where $$\delta_j=\epsilon^{-\frac{1}{k-j+1}}.$$
\begin{figure}
    \centering
   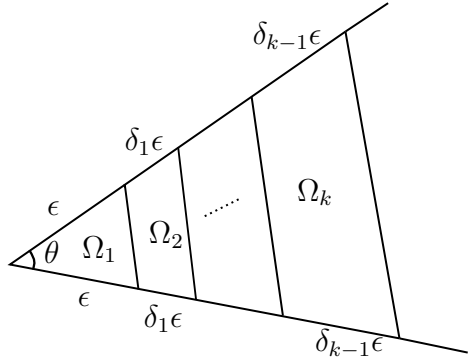
    \caption{$k$-tuple concentrating at a corner.}
    \label{fig}
\end{figure}
A straightforward calculation yields 
$$\eta^\pa(\Om_1)=\sin(\theta/2),\quad \text{and}\quad \eta^\pa(\Omega_j)=\frac{\delta_{j-1}+2\left(1+\delta_1+\ldots+\delta_{j-2}\right)}{\delta_{j-1}}\sin(\theta/2).$$
Consequently, for $j=2,\dots,k$, we have
\begin{equation}\label{ah:eq3}\eta^\pa(\Omega_1)<\eta^\pa(\Omega_j)\end{equation}
and
\begin{equation}\label{ah:eq2}\eta^\pa(\Omega_j)=\sin(\theta/2)+O(\epsilon^{\frac{1}{(k-j+3)(k-j+2)}}).\end{equation}
This entails that
$$I_k(D_n)\le\max_{j=2,\dots,k}\eta^\pa(\Omega_j)\leq\sin(\theta/2)+C \epsilon^{\frac{1}{k(k+1)}}=\cos(\pi/n)+C \epsilon^{\frac{1}{k(k+1)}},$$
which establishes the claimed inequality by the arbitrariness of $\epsilon$.

\smallskip

Next, we show the reverse inequality, i.e. 
\begin{equation}
\label{ineq-2}
I_k(D_n)\geq \cos(\pi/n).
\end{equation}
Let $\tilde\A_k(D_n)$ denote the collection of all $k$-tuple $(\Om_1,\cdots,\Om_k)\in\A_k(D_n)$ such that \begin{itemize}
\item[(1)] {$\etap(\Om_j)<1,~1\le j\le k,$}
    \item[(2)] each $\Om_j$ is connected, and 
    \item[(3)]  the closure of each $\pa_I\Om_j$ consist of disjoint union of straight segments. {In particular, any two connected component of $\pa_I\Om_j$  have no intersection on $\pa D_n$.} 
\end{itemize}

{With a similar argument as in the proof of Theorem\,\ref{balls}, we show that}
\begin{equation}\label{ikdn}
    I_k(D_n)=\ \inf_{(\Om_1,\cdots,\Om_k)\in\tilde\A_k(D_n)}\max_j\etap(\Om_j).
    \end{equation}
{Having inequality \eqref{ineq-1}, w.l.o.g we can assume that all $\Om_j$ satisfy (1).\\ By a verbatim argument as on page \pageref{6}, we can assume that each $\Om_j$ is connected. Assume that the interior boundary of some of $\Om_j$ is not a collection of disjoint straight segments with endpoints on the boundary of $D_n$. We show that we can replace $(\Om_1,\cdots,\Om_k)$ with a new collection $(\tilde\Om_1,\cdots,\tilde\Om_k)\in \tilde\A_k(D_n)$ such that $\max_j\etap(\Om_j)\ge \max_j\etap(\tilde\Om_j)$. 
We consider two cases:
\textit{Case 1.}~ The two endpoints of each connected component of $\pa_I\Om_j$, $1\le j\le k$, belong to distinct edges of $D_n$. Then we can simultaneously substitute all $\Om_j$ by $\tilde\Om_j$ whose interior boundary consists of disjoint union of straight segments. Clearly $(\tilde\Om_1,\cdots,\tilde\Om_k)\in \tilde\A_k(D_n)$ and as explained on page \pageref{6}, we have $\etap(\Om_j) \ge \etap(\tilde\Om_j)$. \\
\begin{figure}
    \centering
\begingroup%
  \makeatletter%
  \providecommand\color[2][]{%
    \errmessage{(Inkscape) Color is used for the text in Inkscape, but the package 'color.sty' is not loaded}%
    \renewcommand\color[2][]{}%
  }%
  \providecommand\transparent[1]{%
    \errmessage{(Inkscape) Transparency is used (non-zero) for the text in Inkscape, but the package 'transparent.sty' is not loaded}%
    \renewcommand\transparent[1]{}%
  }%
  \providecommand\rotatebox[2]{#2}%
  \newcommand*\fsize{\dimexpr\f@size pt\relax}%
  \newcommand*\lineheight[1]{\fontsize{\fsize}{#1\fsize}\selectfont}%
  \ifx\svgwidth\undefined%
    \setlength{\unitlength}{434.0747503bp}%
    \ifx\svgscale\undefined%
      \relax%
    \else%
      \setlength{\unitlength}{\unitlength * \real{\svgscale}}%
    \fi%
  \else%
    \setlength{\unitlength}{\svgwidth}%
  \fi%
  \global\let\svgwidth\undefined%
  \global\let\svgscale\undefined%
  \makeatother%
  \begin{picture}(1,0.32310719)%
    \lineheight{1}%
    \setlength\tabcolsep{0pt}%
    \put(0,0){\includegraphics[width=\unitlength,page=1]{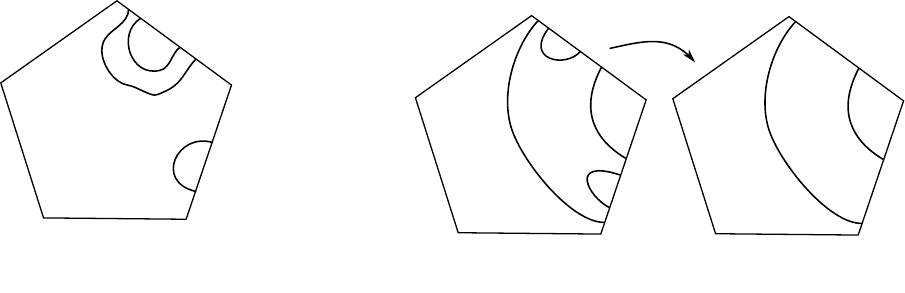}}%
    \put(0.11292253,0.013){\color[rgb]{0,0,0}\makebox(0,0)[lt]{\lineheight{1.25}\smash{\begin{tabular}[t]{l}$(a)$\end{tabular}}}}%
    \put(0.72306557,0.013){\color[rgb]{0,0,0}\makebox(0,0)[lt]{\lineheight{1.25}\smash{\begin{tabular}[t]{l}$(b)$\end{tabular}}}}%
    \put(0,0){\includegraphics[width=\unitlength,page=2]{polygonD.pdf}}%
  \end{picture}%
\endgroup%

    \caption{Domains in (a) do not satisfy property (1). The domain on the left in (b)  can be replace by a new domain (the domain on the right) such that each component of the interior boundary has endpoints on distinct edges.  }
    \label{fig:polygonD}
\end{figure}
\textit{Case 2.}~ Some of the connected components of $\pa_I\Om_j$ have two endpoints on the same edge of $D_n$.  We can immediately rule out the situation when $\pa_E\Omega_j$ only consists of one or two connected components with all endpoints on the same edge, see Figure \ref{fig:polygonD}(a). Because in this case $\etap(\Omega_j)>1$. Hence, since we have more than one domain in the collection, at least  one of the components of  $\pa_I\Omega_j$ has endpoints  on two distinct edges. For each $\Om_j$, we can remove those boundary components whose endpoints are on the same edge and get a new collection of domains in $\mathcal{A}_k(D_n)$ which satisfies the assumption of Case 1, see Figure \ref{fig:polygonD}(b).  This process increases the length of the exterior boundary and decreases the length of the interior boundary of each $\Omega_j$. Now we do the same procedure as in Case 1 to get a new collection $(\tilde\Om_1,\cdots,\tilde\Om_k)\in \tilde\A_k(D_n)$ satisfying  $\etap(\Om_j) \ge \etap(\tilde\Om_j)$. Therefore, we have identity~\eqref{ikdn}.
\\}
To prove inequality \eqref{ineq-2}, we show that any $k$-tuple $(\Om_1,\cdots,\Om_k)\in\tilde\A_k(D_n)$ contains an element $\Omega_j$ such that $\eta^{\pa}(\Omega_j)\geq \cos(\pi/n)$.

\smallskip

For the considerations below, it is convenient to introduce the following notation:
we say that a domain $\Omega\subset D_n$ is of type $\ell$, if
$\pa_E\Omega$ has components in exactly $\ell$ edges of $D_n$.
Notice that each $\Om\in \tilde\A(D_n):=\tilde\A_1(D_n)$ is of type $\ell\ge2$.\\ {Lemma\,\ref{lemma-type} implies that
there are at most $n-2$ disjoint domains $\Omega_j$ which are of type $\ell$ with $\ell \geq 3$.}
Consequently, any $k$-tuple $(\Om_1,\cdots,\Om_k)\in\tilde\A_k(D_n)$ with $k\geq n$, has at least two elements of type $2$. Assume w.l.o.g. that  $\Omega_1$ is of type $2$.
If the boundary components of $\Omega_1$ lie on edges which are not adjacent {(it can be only the case when $n\ge4$)}, then we have $\eta^{\pa}(\Omega_1)\geq 1$ - compare Figure\,\ref{fig4}.
\begin{figure}[ht]
    \centering
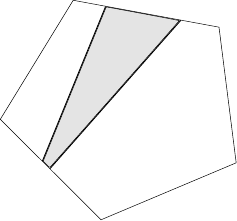
    \caption{Type $2$ domain in $D_5$ with non-adjacent edges.}
    \label{fig4}
\end{figure}
Thus, we assume w.l.o.g. that the boundary components of $\Omega_1$ lie on adjacent edges $e_1$ and $e_2$.
Denote the corner which is the intersection point of the edges $e_1$ and $e_2$ by $C$.
There are now two possible cases: 
either $C\in\pa_E\Om_1$ or $C\notin\pa_E\Om_1$
- compare Figure\,\ref{fig5}.
\begin{figure}[ht]
    \centering
  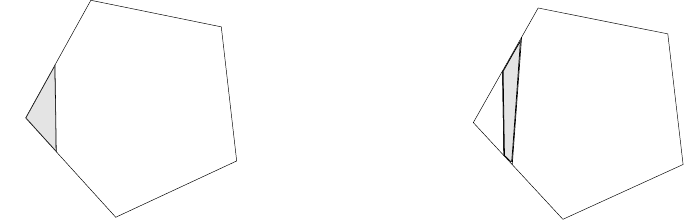
    \caption{$C$ covered by a triangle - $C$ not covered at all.}
    \label{fig5}
\end{figure}
If $C\notin\pa_E\Om_1$ then
there exists a domain $\Omega_j$ of the $k$-tuple $(\Om_1,\cdots,\Om_k)$ which is of type $2$ and closest to $C$. Notice that $\Om_j$ can be equal to $\Om_1$. 
Let $T$ be the triangle which covers $C$ and two of its edges are subset of $e_1$ and $e_2$ and the third one is subset of $\pa_I\Om_j$.
We can then substitute $\Omega_j$ by the union of $\Omega_j$ and $T$
since $\eta^{\pa}(\Omega_j)\geq \eta^{\pa}(\Omega_j\cup T)$ by construction.
Note that $\Omega_j\cup T$ is a triangle covering the corner $C$. An elementary calculation shows $\eta^{\pa}(\Omega_j\cup T)$ attains its minimum $\cos(\pi/n)$ when $\Omega_j\cup T$ is  an isosceles triangle. We conclude that  for any given $k$-tuple $(\Om_1,\cdots,\Om_k)\in \tilde\A_k(D_n)$, we have $$\max_j\eta^\pa(\Om_j)\ge \cos(\pi/n).$$
This establishes inequality (\ref{ineq-2}).

\smallskip

To complete the proof of (i), 
we are left with proving the inequality 
\begin{equation}
\label{ineq}
\cos\left(\pi/n\right)\le I_k(\DD)=\sin(\pi/k)k/\pi
\end{equation}
for all $k\geq n$.
For this purpose, we introduce the function 
$$f:\R^{+}\rightarrow\R,\qquad x\mapsto\sin(\pi/x)/(\pi/x),$$
which 
satisfies $f(k)=I_k(\DD)$ for $k\in\N$.
Since $f$ is strictly increasing it is sufficient to prove the inequality (\ref{ineq}) for $k=n$. This is however equivalent to the inequality $\tan(\pi/n)\geq \pi/n$ which clearly holds for all $n\geq 3$. Thus, $(i)$ is established.

\bigskip

We now turn to the proof of (ii) and start by showing
the inequality 
\begin{equation}
\label{ineq1}    
I_k(D_n)\leq\sin(\pi/k)\cot(\pi/n)k/n.
\end{equation}
It is sufficient to construct a $k$-tuple $(\Omega_1,\cdots,\Om_k)\in \tilde\A_k(D_n)$ such that the inequality
$\max\eta^{\pa}(\Omega_j)\leq\sin(\pi/k)\cot(\pi/n)k/n$ is satisfied.
We inscribe a regular $k$-gon into the $n$-gon as follows: 
consider the midpoints of the edges of the $n$-gon.
 The vertices of the $k$-gon are given by taking every ${m}$-th of the midpoints. {Recall that $n=mk$ by the assumption.} The $k$-gon hence divides $D_n$ into $k+1$ domains.
Those of these domains which cover the corners of $D_n$ yield
a $k$-tuple ${(\Omega_1,\dots,\Omega_k)}\in\A_k(D_n)$.
See Figure\,\ref{fig6} in which the situation is demonstrated for $n=6$ and $k=3$.
\begin{figure}[ht]
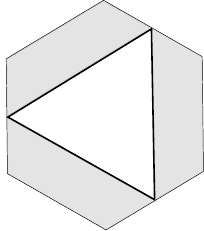
	\caption{$k$-tuple for $n=6$ and $k=3$.}\label{fig6}
\end{figure} 

{All $\Om_j$ are congruent to a domain which we denote  by \label{tildeom}$\tilde\Om$. The exterior boundary  of $\tilde\Om$} is given by ${m}$ times the length of one side of the $n$-gon $D_n$. Straightforward elementary calculations yield 
$${\lvert\pa_E\tilde\Om\lvert}=2{m}\sin(\pi/n)\qquad\mbox{and}\qquad
\lvert{\pa_E\tilde\Om}\lvert=2\cos(\pi/n)\sin(\pi/k).$$
Therefore,
\begin{equation}
   { \eta^\pa(\Om_j)=\eta^\pa(\tilde\Om)}=\sin(\pi/k)\cot(\pi/n)k/n
\end{equation}
and it proves (\ref{ineq1}).
We are thus finally left with proving
$\sin(\pi/k)\cot(\pi/n)k/n\le I_k(\DD)$. This is an immediate consequence of 
$\cot(\pi/n)\pi/n\leq 1$ and Theorem\,\ref{balls}.\\

We now prove the equality case. Let $(\Om_1,\cdots,\Om_k)$ be a $k$-tuple in $\tilde\A_k(D_n)$ such that
\begin{equation}\label{starom}\max_j\eta^{\pa}(\Omega_j)\leq\sin(\pi/k)\cot(\pi/n)k/n.\end{equation} W.l.o.g. we assume that  all vertices are covered. We consider two cases.

\noindent\textit{Case 1.~} Assume that for every $j\in\{1,\cdots,k\}$, the exterior boundary $\extS_j$ is connected. Then there exists at least one $j$ such that $\frac{|\pa D_n|}{n}\le |\extS_j|\le|\pa D_n|/k=|\pa_E\tilde\Om|$.  By Proposition~\ref{mainapp},  and inequalities \eqref{ah1} and \eqref{ah2} in Appendix, we get $\eta^\pa(\Om_j)\ge \sin(\pi/k)\cot(\pi/n)k/n$. Hence, we have  equality in \eqref{starom}. \\   
\noindent\textit{Case 2.~} Assume that there is at least one $j_0$ for which $\extS_{j_0}$ is not connected. {By assumption, all vertices are covered. Hence, in each connected component of  $D_n\setminus\Om_{j_0}$, there is at least one domain $\Omega_j$ with connected $\extS_j$. If  $|\pa_E\Om_j|\le |\pa_E\tilde\Om|$ for some  $\Om_j$, we repeat the same argument as in Case 1. Let us assume that for all $\Om_j$ with connected exterior boundary, $|\pa_E\Om_j|>|\pa_E\tilde\Om|$.     We claim} that the length of each connected component of $\pa_I\Om_{j_0}$is bounded below by $2\cos(\pi/n)\sin(\pi/k)$. {Let $L$ be a connected component of $\pa_I\Om_{j_0}$. It divides $\pa D_n$ into two parts $L_1$ and $L_2$.  Note that in each component  of $D_n\setminus L$, there is  $\Omega_j$ with connected $\pa_E\Om_j$. Hence, $|L_i|>|\pa_E\tilde\Om|$, $i=1,2$. Let $L_1$ be the one with $|L_1|\le|\pa D_n|/2$ and consider the domain $\Om$ enclosed by $L$ and $L_1$. Then we can use Proposition \ref{mainapp} in Appendix (taking $\Om_0=\tilde\Om$)  to get $$|L|=|\pa_I\Om|\ge |\pa_I\tilde\Om|=2\cos(\pi/n)\sin(\pi/k). $$ 
Since $L$ was arbitrary, the above inequality holds for each connected component of $\pa_I\Om_{j_0}$.}
{Let $m\ge2$ be the number of connected components of $\pa_I\Om_{j_0}$.} Then we have $|\intS_{j_0}|\ge 2m\cos(\pi/n)\sin(\pi/k)$. By assumption, $\eta^\pa(\Om_{j_0})\le \sin(\pi/k)\cot(\pi/n)k/n$ which implies
$$|\extS_{j_0}|\ge m|{\pa D_n}|/k.$$ 
As a result we get
$\sum_{j=1}^k|\extS_j|>|\pa D_n|$ which is impossible. 
Therefore, condition \eqref{starom} only holds in Case 1.
This completes the proof of (ii).
\end{proof}

\subsection{Escobar constants $I_k$ for Euclidean and curvilinear polygons}
\label{polygons}
In this subsection we first give upper bounds for $I_k(M)$ where $M$ is a Euclidean $n$-gon, i.e. we prove Theorem\,\ref{intro-polygon}. Afterwards, we generalize these considerations to curvilinear polygons, i.e. we show Theorem\,\ref{curvilinear}.

\smallskip

\begin{proof}[\textbf{Proof of Theorem\,\ref{intro-polygon}}]
The idea of the proof is the same as the proof of inequality \eqref{ineq-1} (see the proof of Theorem\,\ref{regular-polygons}), i.e. the domains of a $k$-tuple will concentrate at the corner with the smallest interior angle.

\smallskip

Clearly, there exists $1\le n_0\le n$ such that $$0<\theta_1\le\theta_2\le\cdots\le\theta_{n_0}<\pi<\theta_{n_0+1}\le\dots\le\theta_n<2\pi.$$
For every $\theta\in\{\theta_j:1\le j\le n_0\}$, we consider the $k$-tuple $(\Omega_1,\cdots, \Om_k)$ illustrated in Figure\,\ref{fig} near the corner of  angle $\theta$.
Consequently, we get the same estimate as in \eqref{ah:eq3} and \eqref{ah:eq2}:
$$\eta^\pa(\Omega_j)=\sin(\theta/2)+O(\epsilon^{\frac{1}{(k-j+3)(k-j+2)}}),\qquad 1<j\le k.$$
This entails that
$$I_k(M)\leq\max_{j=2,\dots,k}\eta^\pa(\Omega_j)\le\sin(\theta/2)+O(\epsilon^{\frac{1}{(k+1)k}}).$$
By the arbitrariness of $\epsilon$, we thus get
$$I_k(M)\leq\max_{j=2,\dots,k}\eta^\pa(\Omega_j)\le\sin(\theta/2).$$
Therefore, we have
$$I_k(M)\le \displaystyle{\min_{1\le j\le n_0}} \sin(\theta_j/2)=\sin(\theta_1/2).$$  Since the sum of interior angles of any $n$-gon is $(n-2)\pi$, the maximum value for $\theta_1$ is attained when $M$ is a regular $n$-gon and it is equal to $\frac{(n-2)\pi}{n}$. 
Consequently, we obtain
$$\sin (\theta_1/2)\le \sin\left(\pi/2-\pi/n\right)=\cos\left(\pi/n\right),$$
which establishes the claim.
\end{proof}

In what follows let $\cP_n$ be the family of $n$-gons.

\begin{corollary}\label{cor1}
For every $M\in \cP_n$, we have
    the inequality
\begin{equation}\label{conj2}I_k(M)\le I_k(D_n)\le I_k(\DD),\qquad \mbox{for all}\quad k\ge n.\end{equation}
\end{corollary}

\begin{remark}
\begin{enumerate}
\item 
The statement of Theorem\,\ref{intro-polygon} can be extended to any polygon in the hyperbolic plane. The proof uses the 
     same $k$-tuple as in the proof of Theorem \ref{intro-polygon} and the Toponogov theorem.
     \item Note that we have  equality in Theorem \ref{intro-polygon} when $M$ is a regular polygon and $k\ge n$.
However, for any fixed $k\ge2$ there is no lower bound for $I_k(M)$, compare Remark\,\ref{rem-1}.
\end{enumerate}
\end{remark}

It is an interesting question whether among all $n$-gons, $I_k$ is maximized by $D_n$ for all $k\ge2$. 
\begin{open}Let $D\in\cP_n$. Does  inequality 
\begin{equation*} 
I_k(D)\le I_k(D_n), 
\end{equation*}
hold {for all} $k\ge 2$?
\end{open}
From Theorem \ref{regular-polygons}, we know that it holds for $k\ge n$ and $k=3,4$.\\

\smallskip

The following theorem  demonstrates that if we fix $M\in\cP_n$ then  we have equality in Theorem \ref{intro-polygon} when \lq $k$ is large enough\rq.

\begin{theorem}
Let $M\in\cP_n$ be given. There exists $n_0\ge n$ such that $I_k(M)=\sin(\theta_1/2)$ for all $k\geq n_0$.
Here $\theta_1$ is given as in Theorem\,\ref{intro-polygon}.
\end{theorem}

\begin{proof}
As in the proof of Theorem\,\ref{regular-polygons}, we say that a domain $\Omega\subset M$ is of type $\ell$, if
{$\pa_E\Omega$} has components in exactly $\ell$ edges of $M$.
{By Lemma\,\ref{lemma-type} there can be at most $n-2$ domains $\Omega_j$ which are of type $\ell$ with $\ell \geq 3$.}
Hence, any $k$-tuple $(\Omega_1,\cdots,\Omega_k)\in A_k(M)$ with $k\geq n$, has at least two elements of type $2$.
Assume w.l.o.g. that $\Omega_1$ is of type $2$.
If the boundary components of $\Omega_1$ lie on adjacent edges $e_1$ and $e_2$, then proceed as in the proof of Theorem\,\ref{regular-polygons} to get
$\eta^{\pa}(\Omega_1)\geq \sin(\theta_j/2)\geq\sin(\theta_1/2),$ where $\theta_j$ denotes the angle enclosed by $e_1$ and $e_2$.
Therefore, we assume w.l.o.g. that the boundary components of $\Omega_1$ lie on non-adjacent edges $e_1$ and $e_2$. Let $d=\mbox{dist}(e_1,e_2)$. Thus the interior boundary of $\Omega_1$ is at least $2d$. When $k$ increases, the number of type $2$ domains necessarily increases. Consequently, there can not exist a uniform lower bound on the lengths of their exterior boundaries.
Hence, for large enough $k$, there has to exist a domain $\Omega_i$ of type $2$ with exterior boundary less than $2d/ \sin(\theta_1/2)$.
Consequently,  $\eta^{\pa}(\Omega_i)\geq \sin(\theta_1/2)$. By Theorem~\ref{intro-polygon}, we know $I_k(M)\le\sin(\theta_1/2)$. This establishes the claim.
\end{proof}

Finally, we generalize the preceding considerations to curvilinear polygons and prove Theorem \ref{curvilinear}. \begin{figure}
    \centering
    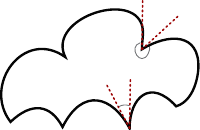
    \caption{A curvilinear 6-gon.}
    \label{ah:curvilinear}
\end{figure}

{ Let us first recall the definition of a curvilinear polygon. Let $M$ be a  simply connected domain in $\R^2$ whose boundary $\partial M$ consists of union $n$-smooth curve and $n$ vertices. The vertices are point were two smooth parts of the boundary meet and the angle of each vertex is in $(0,2\pi)\setminus\{\pi\}$, see Figure \ref{ah:curvilinear}. We call such domain $M$ a \textit{curvilinear $n$-gon}.}

\begin{proof}[\textbf{Proof of Theorem\,\ref{curvilinear}}]
 The strategy is essentially the same as that of
  the proof of Theorem \ref{regular-polygons}.
  Since the boundaries of the polygons are now possibly more general curves than lines, we need to modify the construction a bit.  
  
  \smallskip 

  For every $\varepsilon>0$, there exists a neighborhood $U\subset M$ of the vertex $\mathcal{V}$ with interior angle $\theta_1$ which lies inside a cone $C$ of angle $\theta_1+\varepsilon$ centered at $\mathcal{V}$, see Fig \ref{fig:curvilinear}.
  Let $\mathcal{C}_1,\dots,\mathcal{C}_k$ be the circles centered at 
  $\mathcal{V}$ with radii $\delta_0\epsilon, (\delta_0+\delta_1)\epsilon,\cdots,\sum_{j=0}^k\delta_j\epsilon$ respectively.
  Here $\delta_0=1$,
  $\delta_j:=\epsilon^{-\frac{1}{k-j+1}}$ for $j=1,\dots, k$ and $\epsilon>0$ is chosen small enough such that the intersection with $M$ is contained in $U$. Denote by $E_1$ and $E_2$ the two edges of the cone. Among the intersection points of $M$ and $\mathcal{C}_j$, we choose
   $p_1^j$ and $p_2^j$ to be  the closest point to $E_1$ and $E_2$ respectively.  Further, denote by $\ell_j$  straight line which connects $p_1^j$ and $p_2^j$. We now define the $k$-tuple $(\Om_1,\cdots,\Om_k)$. Let  $\Omega_j$, $1\le j\le k$,   be the intersection of the strip between  $\ell_j$ and $\ell_{j-1}$ and $M$. Note that $\ell_0$ is a point and is equal to $\mathcal{V}$. The interior boundary of $\Omega_j$ consists of union of segments which are subset of $\ell_j\cup\ell_{j-1}$. Hence, $\partial\Omega_j\cap M=|\intS_j|\le |\ell_j|+|\ell_{j-1}|.$ It is clear that the length of  the exterior boundary $|\extS_j|\ge 2\epsilon\delta_{j}$.  
  Also see Fig \ref{fig:curvilinear}, where $\Om_1$ and $\Om_2$ are depicted. The interior boundary of $\Om_2$ is given by the two blue lines.

  \smallskip
  
Next we construct a $k$-tuple $(\Om_1^C,\dots, \Om_k^C)$ for the cone $C$. In analogy to the above construction, the intersection points of the circles $\mathcal{C}_j$ with the cone $C$ give rise to the exterior boundaries of $\Om_j^C$ -- see  Fig \ref{fig:curvilinear}, where they are depicted in red. The exterior boundary of $\Om_j^C$ is then simply given by the Euclidean distance between $\mathcal{C}_j$ and $\mathcal{C}_{j-1}$ which is $2\epsilon\delta_j$. Note that we assume $\mathcal{C}_0=\mathcal{V}$.
Since $U\subset C$,  the interior boundary of $\Om_j$ is smaller {than or equal to that of} $\Om_j^C$.
Further, the exterior boundary of $\Omega_j$ is larger than that of $\Om_j^C$. Hence, $\etap(\Omega_j)\le\etap(\Omega^C_j) $.
  Therefore,
\[I_k(M)\le\sin\left(\frac{\theta_1+\varepsilon}{2}\right),\]
which establishes
the claim, by arbitrariness of $\epsilon$.
\begin{figure}
    \centering
    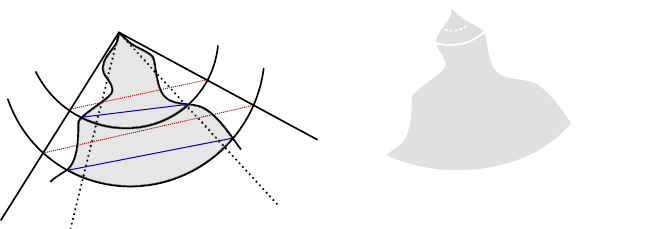
    \caption{$k$-tuple for a curvilinear polygon near a vertex.}
    \label{fig:curvilinear}
\end{figure}
\end{proof}

\smallskip
\section*{Appendix: Calculation of $\eta^\pa$ for a domain in $D_n$}
Consider $\Omega\subset D_n$ with connected exterior boundary $\extS$ of length $L=|\extS|$. We define the \textit{type I symmetrization} of $\Om$ to be the domain $\Om^{\tiny\hexstar}$ with the same exterior length $L$, obtained by the following symmetrization process.
Pick the mid-point $p$ of an edge of $D_n$  and mark two points $A,B$ on the boundary with same distance from $p$ such that the part {$\mathcal{B}_{AB}$} of the boundary of $\pa D_n$ which lies between $A$ and $B$ has length $L$. We denote by $\Om^{\tiny\hexstar}$ the connected domain  whose interior boundary is the segment $AB$ and its exterior boundary is { given by $\mathcal{B}_{AB}$.}
Similarly, we define the \textit{type II symmetrization} of $\Om$ to be a domain $\Om^\star$ when instead of taking $p$ to be the mid-point of an edge, we choose $p$ to be a vertex and $A,B$ on the boundary with same distance from $p$ such that the part of the boundary $\pa D_n$ between $A$ and $B$ has length $L$.
\begin{proposition}\label{propappendix}
Let $\Omega\in \tilde\A(D_n)$ be a domain with connected exterior boundary $\extS$ of length $L=|\extS|\le|\pa D_n|/2$. Let $\Om^{\tiny \hexstar}$  and $\Om^\star$ with $|{\extS^{\tiny \hexstar}}|=|{\extS^\star}|=L$ be the type I and type II symmetrization of $\Om$ respectively. Then we have
\[{|\pa_I\Om|\ge \min\{|\pa_I\Om^{\tiny \hexstar}|,|\pa_I\Om^\star|\}.}\]
In particular,
\[\eta^\pa(\Om)\ge \min\{\eta^\pa(\Om^{\tiny \hexstar}),\eta^\pa(\Om^\star)\}.\]
\end{proposition}

\begin{proof}
Let $P$ and $Q$ be the intersection points of the interior boundary $\intS$ of $\Om$ with $\pa D_n$. W.l.o.g., we assume that $P$ and $Q$ belong to different edges of $D_n$. Let $P$ and $Q$ move clockwise or anticlockwise maintaining distance $L$ from each other along the boundary $\pa D_n$ until they make a symmetric configuration $\Om^{\tiny \hexstar}$ or $\Om^\star$ for a first time. Notice that their displacement $\delta$ will be between $0$ and the length of the half of an edge, see Figure \ref{fig:polygon1a}. We can always choose the moving direction of $P$ and $Q$ so that they do not cross a vertex.  Then it is easy to check that the length of dotted line in Figure \ref{fig:polygon1a} is less than $L$. The resulting domain, after moving $P$ and $Q$ to new positions, is either $\Om^{\tiny \hexstar}$ or $\Om^\star$. Therefore, the statement follows. 
\end{proof}
\begin{figure}
    \centering
    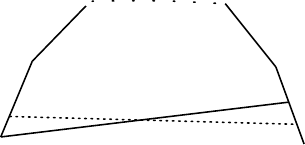
    \caption{Displacement of $P$ and $Q$ to get a symmetric configuration.}
    \label{fig:polygon1a}
\end{figure}
\begin{figure}
    \centering
    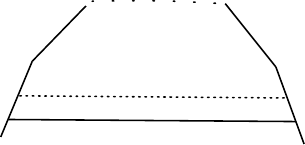
    \caption{}
    \label{fig:polygonb}
\end{figure}
If $|{\extS_1^{\tiny \hexstar}}|=L_1\le L_2=|{\extS_2^{\tiny \hexstar}}|\le|\pa D_n|/2$,
 see Figure \ref{fig:polygonb}, then an easy calculation shows that  
\begin{equation}\label{ah1}\eta^\pa(\Om_2^{\tiny \hexstar})\le\eta^\pa(\Om_1^{\tiny \hexstar}).
\end{equation}
The same holds if we replace ${\tiny \hexstar}$ by $\star$, i.e.
\begin{equation}\label{ah2}\eta^\pa(\Om_2^\star)\le\eta^\pa(\Om_1^\star).
\end{equation}

\begin{remark}
In some cases, we can calculate the minimum of $\eta^\pa(\Om^{\tiny \hexstar})$ and $\eta^\pa(\Om^\star)$ depending on the value of $L$. Let $s:=|\partial D_n|/n$ and $c:=\frac{2s\sin^2(\pi/n)}{1+\cos(\pi/n)-\cos^2(\pi/n)}$. Then
\[\begin{cases}
\eta^\pa(\Om^{\tiny \hexstar})\le\eta^\pa(\Om^\star)& c\le L\le 2s\\
\eta^\pa(\Om^{\tiny \hexstar})\ge\eta^\pa(\Om^\star)& L<c
\end{cases}.\]
Indeed, if $L\le 2s $, then $\Om^{\tiny \hexstar}$ is an isosceles trapezoid and $\Om^\star$ is an isosceles triangle with $\eta^\pa(\Om^\star)=\cos(\pi/n)$. Then we calculate  $\eta^\pa(\Om^{\tiny \hexstar})$ and find $c$ for which the statement above holds.
\end{remark}

\begin{proposition}\label{mainapp}
Let $\Omega_0$ {be a domain in $\tilde{\mathcal{A}}(D_n)$ with}  connected exterior boundary $\extS_0$ of length $s\le L_0\le \frac{sn}{2}$, where $s$ is the edge length of $D_n$. Assume that the endpoints of ${\extS_0^{\tiny \hexstar}}$ or of ${\extS_0^\star}$ position at the midpoint of two edges.  We denote the corresponding symmetrization by $\Om_0^*$. Then for any domain $\Om{\in\tilde{\mathcal{A}}(D_n)}$ with connected exterior boundary of length $L$ 
\begin{itemize}
    \item[(i)] with  $L \le L_0$ we have
\[\eta^\pa(\Om)\ge\eta^\pa( \Om_0^*);\]
{\item[(ii)] with $L_0\le L\le\frac{sn}{2}$ we have
\begin{equation*}\label{om0inq}|\pa_I\Om|\ge |\pa_I\Om_0^*|.\end{equation*}}
\end{itemize}
\end{proposition}
When $n=\ell k$ and  $L_0=\ell s$, then $\Om_0^*=\tilde\Om$, where $\tilde\Om$ is defined on page \pageref{tildeom} and $$\eta^\pa( \Om_0^*)=\sin(\pi/k)\cot(\pi/n)k/n.$$ 
\begin{proof} 
Let assume that the endpoints of ${\extS_0^{\tiny \hexstar}}$ position at the midpoint, i.e. $\Om_0^*=\Om_0^{\tiny \hexstar}$. Thus the endpoints of  ${\extS_0^\star}$ are two vertices of $D_n$. An easy calculation shows that  $\eta^\pa(\Om_0^{\star})\ge \eta^\pa(\Om_0^{{\tiny \hexstar}})$. Similarly, if $\Om_0^*=\Om_0^\star$, then one can show that $\eta^\pa(\Om_0^{\tiny\hexstar})\ge \eta^\pa(\Om_0^{{ \star}})$. Since $L\le L_0$, then we have \eqref{ah1} and \eqref{ah2}. We conclude
\[\min\{\eta^\pa(\Om^{\tiny \hexstar}), \eta^\pa(\Om^\star)\}\ge\min\{\eta^\pa(\Om_0^{{\tiny \hexstar}}), \eta^\pa(\Om_0^{\star})\}=\eta^\pa(\Om_0^*).\]
Combining with Proposition \ref{propappendix} proves the first inequality.\\
{We now prove the second inequality.   By Proposition \ref{propappendix}, we have
\[|\pa_I\Om|\ge \min\{|\pa_I\Om^{\tiny \hexstar}|,|\pa_I\Om^\star|\}|.\]
Since $\frac{s}{2}\le L_0\le L\le\frac{sn}{2}$,
 we get $|\pa_I\Om^\star|\ge|\pa_I\Om_0^\star|$ by considering $\Om^\star$ and $\Om_0^\star$ from a same vertex $p$.  Similarly, we have  $|\pa_I\Om^{\tiny \hexstar}|\ge|\pa_I\Om_0^{\tiny \hexstar}|$, see Figure \ref{fig:app} for an illustration. Therefore,  
 \[|\pa_I\Om|\ge  \min\{|\pa_I\Om_0^{\tiny \hexstar}|,|\pa_I\Om_0^\star|\}=|\pa_I\Om_0^*|.\]
 \begin{figure}
    \centering
\begingroup%
  \makeatletter%
  \providecommand\color[2][]{%
    \errmessage{(Inkscape) Color is used for the text in Inkscape, but the package 'color.sty' is not loaded}%
    \renewcommand\color[2][]{}%
  }%
  \providecommand\transparent[1]{%
    \errmessage{(Inkscape) Transparency is used (non-zero) for the text in Inkscape, but the package 'transparent.sty' is not loaded}%
    \renewcommand\transparent[1]{}%
  }%
  \providecommand\rotatebox[2]{#2}%
  \newcommand*\fsize{\dimexpr\f@size pt\relax}%
  \newcommand*\lineheight[1]{\fontsize{\fsize}{#1\fsize}\selectfont}%
  \ifx\svgwidth\undefined%
    \setlength{\unitlength}{287.62759782bp}%
    \ifx\svgscale\undefined%
      \relax%
    \else%
      \setlength{\unitlength}{\unitlength * \real{\svgscale}}%
    \fi%
  \else%
    \setlength{\unitlength}{\svgwidth}%
  \fi%
  \global\let\svgwidth\undefined%
  \global\let\svgscale\undefined%
  \makeatother%
  \begin{picture}(1,0.48143166)%
    \lineheight{1}%
    \setlength\tabcolsep{0pt}%
    \put(0,0){\includegraphics[width=\unitlength,page=1]{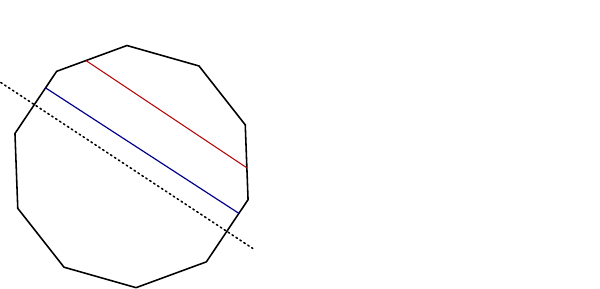}}%
    \put(0.31913042,0.39779549){\color[rgb]{0,0,0}\makebox(0,0)[lt]{\lineheight{1.25}\smash{\begin{tabular}[t]{l}$p$\end{tabular}}}}%
    \put(0,0){\includegraphics[width=\unitlength,page=2]{appendix.pdf}}%
    \put(0.81463689,0.425){\color[rgb]{0,0,0}\makebox(0,0)[lt]{\lineheight{1.25}\smash{\begin{tabular}[t]{l}$p$\end{tabular}}}}%
    \put(0,0){\includegraphics[width=\unitlength,page=3]{appendix.pdf}}%
    \put(0.25999025,0.3150197){\color[rgb]{0,0,0}\makebox(0,0)[lt]{\lineheight{1.25}\smash{\begin{tabular}[t]{l}$\partial\Omega_0^{\star}$\end{tabular}}}}%
    \put(0.24175867,0.23424755){\color[rgb]{0,0,0}\makebox(0,0)[lt]{\lineheight{1.25}\smash{\begin{tabular}[t]{l}$\partial\Omega^{\star}$\end{tabular}}}}%
    \put(0.77883634,0.3465859){\color[rgb]{0,0,0}\makebox(0,0)[lt]{\lineheight{1.25}\smash{\begin{tabular}[t]{l}$\partial\Omega_0^{\tiny\hexstar}$\end{tabular}}}}%
    \put(0.7765532,0.27273659){\color[rgb]{0,0,0}\makebox(0,0)[lt]{\lineheight{1.25}\smash{\begin{tabular}[t]{l}$\partial\Omega^{\tiny\hexstar}$\end{tabular}}}}%
  \end{picture}%
\endgroup%

    \caption{In the notation of Proposition \ref{mainapp}, we consider $\Om_0\subset D_{10}$ with $L_0=3s$. }
    \label{fig:app}
\end{figure}
}
\end{proof}

\end{document}